\documentclass[12pt]{amsart}

\usepackage{amsmath}
\usepackage{amssymb}
\usepackage{amscd}
\usepackage[cmtip,all]{xy}

\title{Galois points for a plane curve and its dual curve, II}
\author{Satoru Fukasawa \& Kei Miura}

\address{Department of Mathematical Sciences, Faculty of Science, Yamagata University,  
Kojirakawa-machi 1-4-12, Yamagata 990-8560, Japan}
\email{s.fukasawa@sci.kj.yamagata-u.ac.jp} 
\thanks{The first author was partially supported by JSPS KAKENHI Grant Number 25800002.} 
\address{Department of Mathematics, Ube National College of Technology, Ube, Yamaguchi 755-8555, Japan}
\email{kmiura@ube-k.ac.jp}
\thanks{The second author was partially supported by JSPS KAKENHI Grant Number 26400057.} 

\newtheorem{theorem}{Theorem}[section]
\newtheorem{proposition}[theorem]{Proposition}
\newtheorem{corollary}[theorem]{Corollary}
\newtheorem{lemma}[theorem]{Lemma} 

\newtheorem{fact}[theorem]{Fact}

\theoremstyle{definition}

\newtheorem{remark}[theorem]{Remark}

\begin{document}
\begin{abstract} 
Let $C \subset \mathbb{P}^2$ be a plane curve of degree at least three. 
A point $P$ in projective plane is said to be Galois if the function field extension induced by the projection $\pi_P: C \dashrightarrow \mathbb P^1$ from $P$ is Galois. 
Further we say that a Galois point is extendable if any birational transformation induced by the Galois group can be extended to a linear transformation of the projective plane.  
This article is the second part of \cite{fukasawa-miura}, where we showed that the Galois group at an extendable Galois point $P$ has a natural action on the dual curve $C^* \subset \mathbb{P}^{2*}$ which preserves the fibers of the projection $\pi_{\overline{P}}$ from a certain point $\overline{P} \in \mathbb{P}^{2*}$.  
In this article we improve such a result, and we investigate the Galois group of $\pi_{\overline{P}}$. 
In particular, we study both when $\overline{P}$ is a Galois point, and when $\deg(\pi_P)$ is prime and $\deg(\pi_{\overline{P}}) = 2\deg(\pi_P)$. 
As an application, we determine the number of points at which the Galois groups are certain fixed groups for the dual curve of a cubic curve.  
\end{abstract}
\maketitle

\section{Introduction}  
Let the base field $K$ be an algebraically closed field of characteristic zero and let $C \subset \mathbb{P}^{2}$ be an irreducible plane curve of degree $d \ge 3$. 
Consider a point $P \in \mathbb{P}^2$ and let $\pi_P: C \dashrightarrow \mathbb{P}^1$ be the projection from $P$. 
Then we denote by $G_P$ and $L_P$ the Galois group and the Galois closure of the function field extension $K(C)/\pi_P^*K(\mathbb{P}^1)$ induced by the map $\pi_P$. 

Under the situation above, Yoshihara introduced the notion of Galois points (cf. \cite{fukasawa, miura-yoshihara1, yoshihara1}). 
We say that $P \in \mathbb{P}^2$ is a {\it Galois point} for $C$ if the field extension $K(C)/\pi_P^*K(\mathbb{P}^1)$ is Galois. 
Moreover, a Galois point $P$ is {\it extendable} if any birational transformation of $C$ induced by the Galois group $G_P$ can be extended to a linear transformation of $\mathbb{P}^2$. 
On the other hand, it follows from \cite{pirola-schlesinger, yoshihara1} that there exists only finitely many points $P \in \mathbb{P}^2\setminus C$ (resp. $P \in C$) such that $G_P$ is not isomorphic to the full symmetric group $S_d$ (resp. $S_{d-1}$), and there are several papers studying those exceptional points (see e.g. \cite{miura1, miura-yoshihara1, miura-yoshihara2}).

Let $\mathbb{P}^{2*}$ be the dual projective plane which parameterizes projective lines of $\mathbb{P}^2$ and let $(X:Y:Z)$ be a system of homogeneous coordinates of $\mathbb{P}^2$. 
We recall that if $F \in K[X, Y, Z]$ is the homogeneous polynomial defining $C$, and ${\rm Sing}(C)$ denotes the singular locus, then the {\it dual map} of $C$ is a rational map $\gamma_C: C \dashrightarrow \mathbb{P}^{2*}$ sending a smooth point $Q \in C \setminus {\rm Sing}(C)$ to the point $(\frac{\partial F}{\partial X}(Q):\frac{\partial F}{\partial Y}(Q):\frac{\partial F}{\partial Z}(Q)) \in \mathbb{P}^{2*}$ parameterizing the projective tangent line $T_QC$ to $C$ at $Q$. 
So the {\it dual curve} $C^*$ of $C$ is the closure of the image of $\gamma_C$, and projective duality holds, that is $C^{**}=C$ (see, for example, \cite{kleiman, namba}). 

In \cite{fukasawa-miura} we connected the study of Galois point for $C$ to the dual curve $C^*$ as follows. 
Given a point $\overline{P} \in \mathbb{P}^{2*}$, we consider the set 
$$G[\overline{P}]:=\{ \tau \in {\rm Bir}(C^*) \ | \ \tau (C^* \cap \ell \setminus \{\overline{P}\}) \subset \ell \ \mbox{ for a general line } \ell \ni \overline{P} \}$$ 
of birational transformation of $C^*$ preserving the fibers of the projection $\pi_{\overline{P}}: C \dashrightarrow \mathbb{P}^1$. 
Then the following holds (see \cite[Proposition 1.5]{fukasawa-miura}).

\begin{proposition} \label{Relation}
Let $C$ be a plane curve with an extendable Galois point $P \in \mathbb{P}^2$.  Then any $\sigma \in G_P$ induces a natural linear transformation $\overline{\sigma}: \mathbb{P}^{2*} \rightarrow \mathbb{P}^{2*}$ (see \cite[Lemma 2.2]{fukasawa-miura}), and there exists a unique point $\overline{P} \in \mathbb{P}^{2*}$ such that the map $\sigma \mapsto \overline{\sigma}$ induces an injective homomorphism 
$$ G_P \hookrightarrow G[\overline{P}]. $$
In particular, the degree of the projection $\pi_{\overline{P}}:C^* \dashrightarrow \mathbb{P}^1$ from $\overline{P}$ is at least the order of $G_P$.  
\end{proposition} 

This article is a continuation of the analysis we performed in \cite{fukasawa-miura}. 
In particular, we complete Proposition \ref{Relation} by connecting the projections $\pi_P$ and $\pi_{\overline{P}}$ through the dual map $\gamma_C$. 
Furthermore, we study the Galois group $G_{\overline{P}}$ and we apply our results to the case of cubic plane curves. 

When $P$ is an {\it outer} Galois point for $C$, that is $P \in \mathbb{P}^2 \setminus C$, we prove the following.  

\begin{theorem} \label{outer}
Let $C$ be a plane curve with an extendable Galois point $P \in \mathbb{P}^2 \setminus C$. 
Then: 
\begin{itemize}
\item[(1)] There exists a morphism $f_P:\mathbb P^1 \rightarrow \mathbb P^1$ such that $f_P \circ \pi_P=\pi_{\overline{P}} \circ \gamma_C$. 
Moreover up to linear transformations of $\mathbb{P}^2$, the defining equation of $C$ is $X^d+G(Y, Z)=0$, $P=(1:0:0)$ and $f_P$ is represented by $(G_Y:G_Z)$.   
\item[(2)] Denoting by $r$ the degree of $f_P$ and by $R$ the degree of the Galois closure of $f_P$, we have that $d \times r \le |G_{\overline{P}}| \le R \times d^r$. 
In particular, $d \le |G_{\overline{P}}| \le (d-1)! \times d^{d-1}$. 
\item[(3)] The point $\overline{P}$ is Galois for $C^*$ if and only if $d=|G_{\overline{P}}|$. 
In this case, $C$ is projectively equivalent to the curve defined by $X^d-Y^eZ^{d-e}=0$ for some $e$. 
\item[(4)] Assume that $r=2$. 
If $\pi_P$ is ramified at a smooth point of $C$, then $|G_{\overline{P}}|=2d^2$. 
\end{itemize} 
\end{theorem} 

On the other hand, when $P$ is an {\it inner} Galois point, i.e. when it is a smooth point of $C$, we deduce the following. 

\begin{theorem} \label{inner}
Let $C$ be a plane curve with an extendable Galois point $P \in C \setminus {\rm Sing}(C)$. 
Then: 
\begin{itemize}
\item[(1)] There exists a morphism $f_P:\mathbb P^1 \rightarrow \mathbb P^1$ such that $f_P \circ \pi_P=\pi_{\overline{P}} \circ \gamma_C$. 
Moreover, up to linear transformations of $\mathbb{P}^2$, the defining equation of $C$ is $X^{d-1}Z+G(Y, Z)=0$, $P=(1:0:0)$ and $f_P$ is represented by $(G_YZ:-G+G_ZZ)$.    
\item[(2)] Denoting by $r$ the degree of $f_P$ and by $R$ the degree of the Galois closure of $f_P$, we have that $(d-1) \times r \le |G_{\overline{P}}| \le R \times (d-1)^r$. 
In particular, $(d-1) \le |G_{\overline{P}}| \le d! \times (d-1)^{d}$. 
\item[(3)] The point $\overline{P}$ is Galois for $C^*$ if and only if $d-1=|G_{\overline{P}}|$. 
In this case, $C$ is projectively equivalent to the curve defined by $X^{d-1}Z-Y^d=0$. 
\item[(4)] If $r=2$, then $|G_{\overline{P}}|=2(d-1)^2$. 
\end{itemize} 
\end{theorem} 

By applying Theorems \ref{outer} and \ref{inner} to the case $|G_{\overline{P}}|=2p^2$ for an odd prime $p$, we achieve the following. 

\begin{theorem} \label{group}
If $d=p$ (resp. $d-1=p$) for an odd prime $p$ and $r=2$ in Theorem \ref{outer} (resp. in Theorem \ref{inner}), then 
$$ G_{\overline{P}} \cong (\mathbb Z/p \mathbb Z) \times D_{2p}, $$
where $D_{2p}$ is the dihedral group of order $2p$. 
\end{theorem}

Finally, suppose that $C$ is a cubic without cusps. 
It is well-known that either $C$ is smooth and $\deg C^*=6$, or $C$ has a single node and $\deg C^*=4$. 
Then we prove the following application to cubic plane curves. 

\begin{theorem} \label{cubic}
Let $C$ be a plane curve of degree $d=3$ which does not have a cusp. 
Then:  
\begin{itemize} 
\item[(1)] There exists a Galois point $R$ for the dual curve $C^*$ only if $C$ is a nodal curve and $R$ is a double point of $C^*$. 
\item[(2)] If $C$ is a nodal curve, then there are exactly three points $R \in \mathbb P^{2*} \setminus C^*$ with $G_R \cong D_8$ for $C^*$. 
\item[(3)] If $C$ is smooth, then there are nine points $R \in \mathbb P^{2*} \setminus C^*$ satisfying $|G_R|=12, 24$ or $48$ for $C^*$. 
\item[(4)] Let $C$ be smooth. 
Then, there exists a point $R \in \mathbb P^{2*} \setminus C^*$ with $G_R \cong (\mathbb Z/3\mathbb Z) \times S_3$ for the dual curve $C^*$ if and only if $C$ is projectively equivalent to the Fermat curve $F_3: X^3+Y^3+Z^3=0$. 
In this case, the number of such points is three. 
\end{itemize}
\end{theorem}

Since there are not so many results on the number of non-Galois points at which the Galois groups are not the full symmetric group (cf. \cite{miura1, miura-yoshihara1, miura-yoshihara2, takahashi}), Theorem \ref{cubic} would give nice examples.  

The article is organized as follows. 
In the next section we present some preliminary results on Galois projections. 
Sections 3 and 4 are devoted to prove Theorems \ref{outer} and \ref{inner}, respectively. 
Finally, Section 5 concerns applications, and hence the proofs of Theorems \ref{group} and \ref{cubic}.

\section{Preliminaries} 
Let $C$ be an irreducible plane curve of degree $d \ge 3$ and let $\pi:\hat{C} \rightarrow C$ be the normalization. 
We denote by $\hat{\pi}_P$ the composite map $\pi_P \circ \pi$.  
When $P \in \mathbb{P}^2 \setminus C$, we define the multiplicity of $C$ at $P$ as zero. 
Let $(U:V:W)$ be a system of homogeneous coordinates of $\mathbb{P}^{2*}$, and let $d^*$ denote the degree of the dual curve $C^*$. 
For two distinct points $P, Q \in \mathbb P^2$, $\overline{PQ} \subset \mathbb P^2$ is the line passing through $P, Q$. 
For a point $Q \in C \setminus {\rm Sing}(C)$, $T_QC \subset \mathbb P^2$ means the (projective) tangent line at $Q$. 
For a projective line $\ell \subset \mathbb P^2$ and a point $Q \in C \cap \ell$, we denote by $I_Q(C, \ell)$ the intersection multiplicity of $C$ and $\ell$ at $Q$.  
For the projection $\hat{\pi}_P$ and a point $\hat{Q} \in \hat{C}$, we denote by $e_{\hat{Q}}$ the ramification index at $\hat{Q}$. 
If $Q \in C \setminus {\rm Sing}(C)$ and $\pi(\hat{Q})=Q$, then we use the same symbol $e_Q$ for $e_{\hat{Q}}$ by abuse of terminology. 
We note the following elementary fact.

\begin{fact} \label{index}
Let $P \in \mathbb P^2$, let $\hat{Q} \in \hat{C}$ and let $\pi(\hat{Q})=Q \ne P$. Then for $\hat{\pi}_P$ we have the following. 
\begin{itemize}
\item[(1)] If $P \in C\setminus {\rm Sing}(C)$, then $e_P=I_P(C, T_PC)-1$.
\item[(2)] Let $h=0$ be a local equation for the line $\overline{PQ}$ in a neighborhood of $Q$. 
Then, $e_{\hat{Q}}={\rm ord}_{\hat{Q}}(\pi^*h)$.  
In particular, if $Q$ is smooth, then $e_Q=I_Q(C, \overline{PQ})$. 
\end{itemize} 
\end{fact}

If the field extension $K(C)/\theta^*K(C')$ induced by a surjective morphism $\theta: C \rightarrow C'$ between smooth curves $C, C'$ is Galois, then the Galois group $G$ acts on $C$ naturally. 
We denote by $G(Q)$ the stabilizer group of $Q$. 
The following fact holds in this case (\cite[III. 7.2, 8.2]{stichtenoth}).

\begin{fact} \label{Galois covering} 
Let $\theta: C \rightarrow C'$ be a surjective morphism of smooth projective curves such that the induced field extension is Galois, and let $G$ be the Galois group. 
For a point $Q \in C$, $e_Q$ means the ramification index at $Q$. 
Then, we have the following.
\begin{itemize}
\item[(1)] The order of $G(Q)$ is equal to $e_Q$ at $Q$ for any point $Q \in C$. 
\item[(2)] Let $Q_1, Q_2 \in C$. If $\theta(Q_1)=\theta(Q_2)$, then $e_{Q_1}=e_{Q_2}$. 
\end{itemize} 
\end{fact}

We often use the {\it standard form} of the defining equation for a plane curve with an extendable Galois point, which is given by the following (see \cite{miura2, yoshihara1, yoshihara2}). 

\begin{proposition} \label{Standard form} 
Let $P \in \mathbb{P}^2$ be a point with multiplicity $m \ge 0$. 
The point $P$ is extendable Galois for $C$ if and only if there exists a linear transformation $\phi$ on $\mathbb{P}^2$ such that $\phi(P)=(1:0:0)$ and $\phi(C)$ is given by 
$$ X^{d-m}G_{m}(Y, Z)+G_d(Y, Z)=0, $$
where $G_i(Y, Z)$ is a homogeneous polynomial of degree $i$ in variables $Y, Z$. 
In this case, the Galois group $G_{\phi(P)}$ is cyclic and there exists a primitive $(d-m)$-th root $\zeta$ of unity such that a generator $\sigma \in G_{\phi(P)}$ is represented by $\sigma(X:Y:Z)=(\zeta X:Y:Z)$. 
\end{proposition} 

\begin{corollary} \label{ramification}
Let $P \in \mathbb{P}^2 \setminus {\rm Sing}(C)$ be extendable Galois. 
For the standard form as in Proposition \ref{Standard form}, if $\hat{Q} \in \hat{C}$ is a ramification point for $\hat{\pi}_P$, then $\pi(\hat{Q})=P$ or $\pi(\hat{Q})$ is contained in the line given by $X=0$. 
\end{corollary}

\begin{proof} 
Let $\hat{Q} \in \hat{C}$ be a ramification point for $\hat{\pi}_P$ and $Q=\pi(\hat{Q})$. 
By Proposition \ref{Standard form}, $C$ is defined by $X^d+G(Y, Z)=0$ (resp. $X^{d-1}Z+G(Y, Z)=0$) if $P \in \mathbb{P}^2 \setminus C$ (resp. if $P \in C$). 
If $P \in C$ and $Q \in \{Z=0\}$, then $P=Q$, by the defining equation $X^{d-1}Z+G=0$. 
Therefore, we may assume that $Q \not\in \{Z=0\}$. 
Since $(X^d+G)_X=dX^{d-1}$ and $(X^{d-1}Z+G)_X=(d-1)X^{d-2}Z$, ${\rm Sing}(C)$ is contained in the set $\{X=0\}$ (resp. $\{XZ=0\}$) if $P \in \mathbb{P}^2 \setminus C$ (resp. if $P \in C$). 
If $Q \in {\rm Sing}(C)$, then $Q \in \{X=0\}$. 
Assume $Q \in C \setminus {\rm Sing}(C)$. 
Note that for any $\sigma \in G_P\setminus \{1\}$, the set $F[P]:=\{R \in \mathbb{P}^2 \ | \ \sigma(R)=R \}$ is equal to the set $\{X=0\} \cup \{P\}$. 
By Fact \ref{Galois covering}(1), $Q \in F[P]$. 
Therefore, $Q \in \{X=0\}$. 
\end{proof}

\section{Extendable outer Galois points} 
In order to prove Theorem \ref{outer}, we present two preliminary lemmas.
We assume that $P \in \mathbb{P}^2 \setminus C$ is an extendable Galois point, and we define $C$ using the standard form of Proposition \ref{Standard form}. 
Hence $P=(1:0:0)$, the defining equation of $C$ is $X^d+G(Y, Z)=0$, and the projection $\pi_P: C \rightarrow \mathbb{P}^1$ is given by $(X:Y:Z) \mapsto (Y:Z)$. 

\begin{lemma} \label{Property of f_P}
Let $G=\prod_{i=1}^{n}(a_iY+b_iZ)^{\ell_i}$ with $a_ib_j-a_jb_i \ne 0$ if $i \ne j$, let $Q_i \in C$ be the point defined by $X=a_iY+b_iZ=0$ and $R_i=\pi_P(Q_i)=(b_i:-a_i) \in \mathbb P^1$. 
Let $f_P:\mathbb{P}^1 \dashrightarrow \mathbb{P}^1$ be the rational map given by $f_P=(G_Y:G_Z)$ and let $r$ be its degree. 
Up to resolving the indeterminacy locus of $f_P$, we have the following:  
\begin{itemize}
\item[(1)] $f_P(R_i)=(a:b)$ and $f_P$ is unramified at $R_i$. 
In particular, $f_P(R_i) \ne f_P(R_j)$ if $i \ne j$. 
\item[(2)] $r=n-1$. 
\end{itemize}
\end{lemma} 

\begin{proof}
Let $G=(aY+bZ)^\ell H$ with $H(b, -a) \ne 0$. 
Then,
$$G_Y=\ell(aY+bZ)^{\ell-1}aH+(aY+bZ)^\ell H_Y, \ G_Z=\ell(aY+bZ)^{\ell-1}bH+(aY+bZ)^\ell H_Z. $$
Therefore, 
$$ f_P=(\ell aH+(aY+bZ)H_Y:\ell bH+(aY+bZ)H_Z). $$
Hence $f_P$ is well-defined at $(b:-a) \in \mathbb{P}^1$, and we have $f_P((b:-a))=(\ell a H(b, -a):\ell b H(b, -a))=(a:b)$. 
Assume that $b \ne 0$. 
Then, 
$$\frac{\ell aH+(aY+bZ)H_Y}{\ell bH+(aY+bZ)H_Z}-\frac{a}{b}=\frac{(aY+bZ)(bH_Y-aH_Z)}{b(\ell bH+(aY+bZ)H_Z)}.$$
Since $bH_Y(b,-a)-aH_Z(b,-a)=(\deg H)H(b,-a) \ne 0$ by Euler formula, the ramification index at $(b:-a)$ is one. 

We consider (2). 
By the description of $f_P$ as above, for each $i$, $G_Y, G_Z$ are divisible by $(a_iY+b_iZ)^{\ell_i-1}$ and not divisible by $(a_iY+b_iZ)^{\ell_i}$. 
Therefore, we have
$$ r=(d-1)-\sum_{i=1}^n(\ell_i-1)=d-\sum_{i=1}^n\ell_i+n-1=n-1. $$
\end{proof}

\begin{lemma} \label{Property of projection from P} 
Using the same notation as in Lemma \ref{Property of f_P}, we have the following. 
\begin{itemize}
\item[(1)] The multiplicity at $Q_i$ is $\ell_i$. 
\item[(2)] Let $e_i$ be the ramification index for $\hat{\pi}_P$ at each point in $\pi^{-1}(Q_i)$ (see Fact \ref{Galois covering}(2)). 
Then, $e_i \ge d/\ell_i$. 
In particular, each point of $\pi^{-1}(Q_i)$ is a ramification point. 
\item[(3)] $\sum_{i=1}^{n}1/e_i \le 1$.  
\end{itemize}
\end{lemma}

\begin{proof}
Assertion (1) is obvious by considering the defining equation. 
Assertion (2) is derived from that $\overline{PQ_i}$ is given by $a_iY+b_iZ=0$ and $C \cap \overline{PQ_i}=\{Q_i\}$. 
By (2), we have
$$ d=\sum_{i=1}^{n}\ell_i \ge \sum_{i=1}^{n}\frac{d}{e_i}=d\sum_{i=1}^{n}\frac{1}{e_i}. $$
\end{proof}

\begin{proof}[Proof of Theorem \ref{outer}] 
The dual map $\gamma_C$ is given by $(dX^{d-1}:G_Y:G_Z)$. 
It follows from the proof of \cite[Lemma 2.4]{fukasawa-miura} that $\overline{P}=(1:0:0)$. 
Since $P=(1:0:0)$ and $\overline{P}=(1:0:0)$, $\pi_P(X:Y:Z)=(Y:Z)$ and $\pi_{\overline{P}}(U:V:W)=(V:W)$. 
Therefore, $(\pi_{\overline{P}} \circ \gamma_C)(X:Y:Z)=(G_Y:G_Z)$. 
If we take $f_P=(G_Y:G_Z)$, then $f_P \circ \pi_P=\pi_{\overline{P}} \circ \gamma_C$. 
Therefore, we have (1). 

We consider (2). 
Let $x=X/Z$, $y=Y/Z$ and let $u=U/W$, $v=V/W$. 
Since $\gamma_C$ is birational (\cite{kleiman}, \cite[Theorem 1.5.3]{namba}), we have $K(x,y)=K(C)=K(C^*)=K(u,v)$ via the dual map $\gamma_C$. 
Let $G_P=\langle \sigma \rangle$ and let $\sigma(X:Y:Z)=(\zeta X:Y:Z)$, as in Proposition \ref{Standard form}. 
Then, $\overline{\sigma}:C^*\rightarrow C^*$ is given by $\overline{\sigma}(U:V:W)=(\zeta^{-1} U:V:W)$ (see \cite[Lemma 2.2]{fukasawa-miura}). 
Therefore, $u^d \in K(u,v)^{\langle \overline{\sigma} \rangle}=K(x,y)^{G_P}=K(y)$. 
Considering the degree of fields extension, we have $K(u^d, v)=K(y)$. 
Then, we have the following diagrams. 
\begin{equation*} 
\xymatrix{C \ar@{-->}[r]^{\gamma_C} \ar[d]_{\pi_P} & 
C^* \ar@{-->}[dd]^{\pi_{\overline{P}}} \\
\mathbb P^1 \ar[rd]_{f_P} & \\
& \mathbb P^1}
\
\
\
\xymatrix{K(x,y) \ar@{=}[r] \ar@{-}[d] & K(u, v) \ar@{-}[d] \\
K(y) \ar@{=}[r] \ar@{-}[rd] & K(u^d, v) \ar@{-}[d] \\
& K(v)}
\end{equation*} 
Since $[K(x,y):K(y)]=d$ and $[K(y):K(v)]=r$, we have $K[(u,v):K(v)]=d \times r$. 
Therefore, $d \times r \le |G_{\overline{P}}|$. 
Let $g(T) \in K(v)[T]$ be the minimal polynomial of $u^d$ of degree $r$ over $K(v)$. 
Then, $g(T^d)$ is the minimal polynomial of $u$ and the set of all roots of $g(T^d)$ is equal to the set $\{\sqrt[d]{\alpha}:g(\alpha)=0\}$. 
Therefore, we have $|G_{\overline{P}}| \le R \times d^r$, where $R$ is the degree of the Galois closure of $f_P$. 

We consider (3). 
The if-part is obvious. 
Let $\overline{P}$ be Galois. 
Assume by contradiction that $r \ge 2$. 
We use the same notation as in Lemmas \ref{Property of f_P} and \ref{Property of projection from P}. 
By Lemma \ref{Property of f_P}(1), $f_P$ is unramified at $R_i$ for $i=1, \ldots, n$. 
By Fact \ref{Galois covering}(2) and Lemma \ref{Property of projection from P}(2), the ramification index at each point in $(f_P\circ\hat{\pi}_P)^{-1}(f_P(R_i))$ is equal to $e_i$. 
By Corollary \ref{ramification} and Lemma \ref{Property of f_P}(1), the ramification index at each point in $(f_P\circ\hat{\pi}_P)^{-1}(f_P(R_i))\setminus \pi^{-1}(Q_i)$ comes from ramification points in $f_P^{-1}(f_P(R_i))$. 
Thereofore, $f_P$ is ramified at points in $f_P^{-1}(f_P(R_i)) \setminus \{R_i\}$ with index $e_i$. 
By Hurwitz formula for $f_P$, we have $-2+2\deg (f_P)=\deg B$, where $B$ is the ramification divisor.  
Since $\deg(f_P)=r$, the ramification index at $R_i$ is $1$, and the ramification index at each point in $f_P^{-1}(f_P(R_i))$ other than $R_i$ is $e_i$, we have 
$$ 2r-2 \ge \sum_{i=1}^{n}\frac{r-1}{e_i}(e_i-1). $$ 
Using Lemmas \ref{Property of projection from P}(3) and \ref{Property of f_P}(2), we have 
$$ \sum_{i=1}^{n}\frac{r-1}{e_i}(e_i-1)=(r-1)\left\{n-\sum_{i=1}^n\frac{1}{e_i}\right\} \ge (r-1)r. $$
Then, $r=2$ and $e_i=1$. 
This is a contradiction. 
Therefore, $r=1$. 
By Lemma \ref{Property of f_P}(2), we have $n=2$. 
Then, $C$ is projectively equivalent to the curve defined by $X^d-Y^eZ^{d-e}=0$. 

We consider (4). 
Assume that $r=2$. 
Then, $R=2$. 
We have $2d \le |G_{\overline{P}}| \le 2d^2$ by (2). 
By assertion (3), we have $2d<|G_{\overline{P}}|\le 2d^2$. 
Let $S \in \hat{C}$ be a point at which $\hat{\pi}_P$ is totally ramified (i.e. $e_S=d$). 
It follows from Corollary \ref{ramification} that $\pi(S)$ is contained in $\{X=0\}$. 
Since $f_P$ is unramified at two points in $f_P^{-1}(f_P(\hat{\pi}_P(S)))$ and $\hat{\pi}_P$ is unramified at each point in $(f_P \circ \hat{\pi}_P)^{-1}(f_P(\hat{\pi}_P(S)))$ other than $S$ by Lemma \ref{Property of f_P}(1), the Galois closure of $f_P$ is given by an extension of $K(C)$ of degree at least $d$. 
Therefore, $|G_{\overline{P}}|=2d^2$. 
When $Q \in (C \setminus {\rm Sing}(C)) \cap \{X=0\}$, then $e_Q=d$ for $\pi_P$.  
\end{proof}

\begin{remark} \label{f_P Galois}
If $f_P$ gives a Galois extension, we have a sharper bound: $|G_{\overline{P}}| \le r \times d^{r}$. 
For example, if $C$ is the Fermat curve $F_d: X^d+Y^d+Z^d=0$ of degree $d$, then $f_P=(Y^{d-1}:Z^{d-1})$, $r=R=d-1$ and we have $|G_{\overline{P}}| \le (d-1) \times d^{d-1}$. 
\end{remark}

\begin{remark} 
Assertion (4) ``$|G_{\overline{P}}|=2d^2$'' in Theorem \ref{outer} holds also under the assumption that $r=2$ and $\hat{\pi}_P$ is totally ramified at a point of $\hat{C}$ (i.e. the ramification index at the point is $d$). 
For example, when $d$ is prime. 
\end{remark} 

\section{Extendable inner Galois points}
In order to prove Theorem \ref{inner}, we present a preliminary lemma.
We assume that $P \in C\setminus {\rm Sing}(C)$ is an extendable Galois point, and we define $C$ using the standard form of Proposition \ref{Standard form}. 
Hence $P=(1:0:0)$, the defining equation of $C$ is $X^{d-1}Z+G(Y, Z)=0$, and the projection $\pi_P: C \rightarrow \mathbb{P}^1$ is given by $(X:Y:Z) \mapsto (Y:Z)$.

\begin{lemma} \label{Property of f_P inner}
Let $G=c\prod_{i=1}^{n}(Y+b_iZ)^{\ell_i}$ with $b_i \ne b_j$ if $i \ne j$, let $Q_i \in C$ be the point defined by $X=Y+b_iZ=0$ and $R_i=\pi_P(Q_i)=(-b_i:1) \in \mathbb P^1$. 
If we take $f_P=(G_YZ:-G+G_ZZ)$, we have the following. 
\begin{itemize}
\item[(1)] $f_P(R_i)=(1:b_i)$. 
\item[(2)] $r=n$. 
\item[(3)] $\pi_P(P)=(1:0)$, $f_P(1:0)=(0:1)$ and $f_P$ is unramified at $\pi_P(P)$. 
\end{itemize}
\end{lemma}

\begin{proof}
Let $G=(Y+bZ)^\ell H$ with $H(-b, 1) \ne 0$. 
Then,
$$G_Y=\ell(Y+bZ)^{\ell-1}H+(Y+bZ)^\ell H_Y, \ G_Z=\ell(Y+bZ)^{\ell-1}bH+(Y+bZ)^\ell H_Z. $$
Therefore, 
$$ f_P=(\ell HZ+(Y+bZ)H_YZ:-(Y+bZ)H+\ell bHZ+(Y+bZ)H_ZZ). $$
We have $f_P((-b:1))=(\ell H(-b, 1):\ell b H(-b, 1))=(1:b)$.  
We have assertion (1). 

We consider (2). 
By the description of $f_P$ as above, we have
$$ r=d-\sum_{i=1}^{n}(\ell_i-1)=d-\sum_{i=1}^n\ell_i+n=n. $$

We consider (3). 
Let $\tilde{y}=Y/X$ and $\tilde{z}=Z/X$. 
Then, $\pi_P((1:\tilde{y}:\tilde{z}))=(\tilde{y}:\tilde{z})=(\tilde{y}:G(\tilde{y}, \tilde{z}))=(1:G(\tilde{y}, \tilde{z})/\tilde{y})$.  
Since the tangent line at $P$ is defined by $Z=0$, $\tilde{y}$ is a local parameter at $P$ and $\pi_P(P)=(1:0)$. 
Then, $f_P(1:0)=(0:1)$. 
Since $G_YZ$ is divisible by $Z$ and $G_Y(1, 0) \ne 0$, the ramification index at $\pi_P(P)=(1:0)$ is one. 
\end{proof}

\begin{proof}[Proof of Theorem \ref{inner}] 
It follows from the proof of \cite[Lemma 2.4]{fukasawa-miura} that $\overline{P}=(1:0:0)$. 
The dual map $\gamma_C$ is given by 
\begin{eqnarray*}
& &((d-1)X^{d-2}Z:G_Y:X^{d-1}+G_Z) \\
&=&((d-1)X^{d-2}Z^2:G_YZ:X^{d-1}Z+G_ZZ) \\
&=&((d-1)X^{d-2}Z^2:G_YZ:-G+G_ZZ). 
\end{eqnarray*} 
Since $P=(1:0:0)$ and $\overline{P}=(1:0:0)$, $\pi_P(X:Y:Z)=(Y:Z)$ and $\pi_{\overline{P}}(U:V:W)=(V:W)$. 
Therefore, $(\pi_{\overline{P}} \circ \gamma_C)(X:Y:Z)=(G_YZ:-G+G_ZZ)$. 
If we take $f_P=(G_YZ:-G+G_ZZ)$, then $f_P \circ \pi_P=\pi_{\overline{P}} \circ \gamma_C$. 
Therefore, we have (1). 
Similarly to the proof of Theorem \ref{outer}(2), we have (2). 

We consider (3). 
The if-part is obvious. 
Let $\overline{P}$ be Galois. 
Assume by contradiction that  $r \ge 2$.  
Note that $\hat{\pi}_P$ is ramified at $\hat{P}$ with index $d-1$, where $\pi(\hat{P})=P$. 
We use the same notation as in Lemmas \ref{Property of f_P inner}. 
By Lemma \ref{Property of f_P inner}(3), $f_P$ is unramified at $\hat{\pi}_P(\hat{P})$. 
By Fact \ref{Galois covering}(2), there exists a point in $(f_P\circ\hat{\pi}_P)^{-1}(f_P(\hat{\pi}_P(\hat{P})))\setminus \{\hat{P}\}$ at which the ramification index is equal to $d-1$. 
By Corollary \ref{ramification} and Lemma \ref{Property of f_P inner}(1), such a ramification point comes from a ramification point in $f_P^{-1}(f_P(\hat{\pi}_P(\hat{P})))\setminus \{\hat{\pi}_P(\hat{P})\}$ at which the ramification index is equal to $d-1$ for $f_P$.
Then, we have $r=d-1+1=d$. 
By Lemma \ref{Property of f_P inner}(2), we have $n=d$ and $C$ is smooth.  
This is a contradiction to \cite[Theorem 1.2]{fukasawa-miura}. 
Therefore, $r=1$. 
By Lemma \ref{Property of f_P inner}(2), we have $n=1$. 
Then, $C$ is projectively equivalent to the curve defined by $X^{d-1}Z-Y^d=0$. 

We consider (4). 
Assume that $r=2$. 
Then, $R=2$. 
We have $2(d-1) \le |G_{\overline{P}}| \le 2(d-1)^2$ by (2). 
By assertion (3), we have $2(d-1)<|G_{\overline{P}}|\le 2(d-1)^2$. 
Since $e_P=d$ for $\pi_P$ and $f_P$ is unramified at two points in $f_P^{-1}(f_P(\pi_P(P))$ by Lemma \ref{Property of f_P inner}(3), the Galois closure of $f_P$ is given by an extension of $K(C)$ of degree at least $d-1$. 
Therefore, $|G_{\overline{P}}|=2(d-1)^2$. 
\end{proof}

\section{Applications}
In this section we present some applications of Theorems \ref{outer} and \ref{inner}. 
In particular, we prove Theorems \ref{group} and \ref{cubic}. 

\begin{proof}[Proof of Theorem \ref{group}] 
It follows from Theorems \ref{outer} and \ref{inner} that $|G_{\overline{P}}|=2p^2$. 
By Sylow's theorem, there exists an element of $G_{\overline{P}}$ of order two.  
Let $\tau$ be such an element. 
Then, we have the splitting exact sequence
$$ 0 \rightarrow {\rm Gal}(L_{\overline{P}}/K(u^p, v)) \rightarrow G_{\overline{P}} \rightarrow \mathbb Z/2\mathbb Z \rightarrow 0, $$ 
by considering the orders of groups. 
Since the group of order $p^2$ is Abelian (see, for example, \cite[p. 27]{suzuki}), the Galois group ${\rm Gal}(L_{\overline{P}}/K(u^p, v))$ is $(\mathbb Z/p^2\mathbb Z)$ or $(\mathbb Z/p\mathbb Z)^{\oplus 2}$. 
Since $K(u,v)/K(v)$ is not Galois, the group ${\rm Gal}(L_{\overline{P}}/K(u,v))$ is not a normal subgroup. 
Then, $\tau^{-1}{\rm Gal}(L_{\overline{P}}/K(u,v))\tau \ne {\rm Gal}(L_{\overline{P}}/K(u,v))$ and hence we have at least two subgroup of ${\rm Gal}(L_{\overline{P}}/K(u^p, v))$ of order $p$. 
Then, we have ${\rm Gal}(L_{\overline{P}}/K(u^p, v)) \cong (\mathbb Z/p\mathbb Z)^{\oplus 2}$. 
Let ${\rm Gal}(L_{\overline{P}}/K(u, v))=\langle \sigma \rangle$.  
Then $\sigma(\tau^{-1}\sigma\tau) \ne 1$ is fixed by the action of $\tau$, that is $\tau^{-1}(\sigma(\tau^{-1}\sigma\tau))\tau=\sigma(\tau^{-1}\sigma\tau)$. 
Let 
$$ A_{\tau}=\left(
\begin{array}{cc}
1 & a \\
0 & b \\
\end{array}
\right)
$$
be the matrix representing $\tau$ as an action on the vector space ${\rm Gal}(L_{\overline{P}}/K(u^p, v)) \cong (\mathbb Z/p\mathbb Z)^{\oplus 2}$ over $\mathbb Z/p\mathbb Z$. 
Since $A_{\tau}^2=1$, $b=-1$ and $A_{\tau}$ is diagonalizable. 
Therefore, we have $G_{\overline{P}} \cong (\mathbb Z/p\mathbb Z)\times (\mathbb Z/p\mathbb Z \rtimes \mathbb Z/2\mathbb Z)$.
\end{proof}

Next we check the following natural property. 

\begin{proposition}
Let $P_1, P_2 \in \mathbb P^2 \setminus {\rm Sing}(C)$ be extendable Galois points. 
If $\overline{P_1}=\overline{P_2}$, then $P_1=P_2$.  
\end{proposition}

\begin{proof}
Assume that $\overline{P_1}=\overline{P_2}$. 
Let $\overline{G_{P_i}}$ be the image of $G_{P_i} \hookrightarrow G[\overline{P_i}]$ for $i=1,2$. 
We consider the intersection $\overline{G_{P_1}} \cap \overline{G_{P_2}}$ in $G[\overline{P_1}]=G[\overline{P_2}]$. 
If there exists $\overline{\sigma} \in \overline{G_{P_1}} \cap \overline{G_{P_2}}\setminus \{1\}$, then we have $P_1=P_2$ by considering the fixed locus of $\sigma$. 
Assume $\overline{G_{P_1}} \cap \overline{G_{P_2}}=\{1\}$. 
Let $\sigma, \tau$ be generators of $G_{P_1}$, $G_{P_2}$ respectively. 
Since $|\langle \overline{\sigma}, \overline{\tau} \rangle| \ge (d-1)^2$ and $|G[\overline{P_1}]| \le d(d-1)$, $\langle \overline{\sigma}, \overline{\tau} \rangle=G[\overline{P_1}]$ and $\overline{P_1}$ is Galois. 
By Theorem \ref{outer}(3) and \ref{inner}(3), we have $|G[P_1]| \le d$. 
This is a contradiction. 
\end{proof}

For smooth curves, we have the following proposition which relies also on \cite{miura-yoshihara1, yoshihara1}. 
\begin{proposition} \label{smooth} 
If $C$ is smooth and $P$ is extendable Galois, then $d^*=d(d-1)$ and $\overline{P} \in \mathbb P^{2*} \setminus C^*$. 
Furthermore, we have the following for smooth curves with two or more Galois points. 
\begin{itemize}
\item[(1)] If $C$ is defined by $X^3Z+Y^4+Z^4=0$, then there exist a point $\overline{P} \in \mathbb P^{2*} \setminus C^*$ with $|G_{\overline{P}}| \le 6 \times 4^3$ and four points $\overline{P} \in \mathbb P^{2*} \setminus C^*$ with $|G_{\overline{P}}| \le 24 \times 3^4$. 
\item[(2)] If $C$ is defined by $X^{d-1}Z+Y^d+Z^d=0$ ($d \ge 5$), then there exists a point $\overline{P} \in \mathbb P^{2*} \setminus C^*$ with $|G_{\overline{P}}| \le (d-1)! \times d^{(d-1)}$ and a point $\overline{P} \in \mathbb P^{2*} \setminus C^*$ with $|G_{\overline{P}}| \le d! \times (d-1)^d$. 
\item[(3)] If $C$ is the Fermat curve $F_d$, then there exist three points $\overline{P} \in \mathbb P^{2*} \setminus F_d^*$ with $|G_{\overline{P}}| \le (d-1) \times d^{(d-1)}$.  
\end{itemize}
\end{proposition}

\begin{proof}
Since $C$ is smooth, we have $\deg C^*=d(d-1)$ and $n=d$. 
It follows from Lemmas \ref{Property of f_P}(2) and \ref{Property of f_P inner}(2) that $r=d-1$ and $r=d$ if $P \in \mathbb P^2 \setminus C$ and if $P \in C$ respectively. 
Then the degree of $f_P\circ\pi_P$ is $d(d-1)$. 
Therefore, by Theorems \ref{outer}(1) and \ref{inner}(1), the degree of $\pi_{\overline{P}}$ is $d(d-1)$ and $\overline{P} \in \mathbb P^{2*} \setminus C^*$. 

For the Fermat curve $F_d$ and the outer Galois point $P=(1:0:0)$, $f_P=(Y^{d-1}:Z^{d-1})$ and $r=R=d-1$. 
\end{proof}

To prove Theorem \ref{cubic}, we would like to determine lines $\ell \subset \mathbb{P}^{2*}$ such that $\ell \cap {\rm Sing}(C^*)\ne \emptyset$ for a cubic curve $C$ which does not have a cusp.
Let ${\rm Flex}(C) \subset C$ be the set of all flexes of $C$. 
Then $\gamma_C({\rm Flex}(C))={\rm Sing}(C^*)$ as $C$ does not admit bitangent lines. 
The conditions appearing below will represent types of ramification indices for the projection $\hat{\pi}_R$ from a point $R \in \mathbb P^{2*}$. 

\begin{lemma} \label{nodal}
Let $C$ be a nodal curve of degree three and let $\ell \subset \mathbb{P}^{2*}$ be a line with $\ell \cap {\rm Sing}(C^*)\ne \emptyset$. 
Then, one of the following conditions holds. 
\begin{itemize}
\item[$(1)$] $C^* \cap \ell$ consists of two singular points.  
\item[$(2)$] $C^* \cap \ell$ consists of one singular point and two smooth points. 
\item[$(3)$] $C^* \cap \ell$ consists of exactly one singular point $Q_1$ and one smooth point $Q_2$ with $Q_1 \not\in T_{Q_2}C^*$. 
\end{itemize} 
On the other hand, there exists a unique line $\ell \subset \mathbb P^{2*}$ such that $C^* \cap \ell$ consists of exactly two points at which the tangent lines are $\ell$. 
\end{lemma}

\begin{proof} 
Now $d^*=4$. 
We recall the well-known fact that there are three flexes for $C$. 
Let $Q \in {\rm Flex}(C)$ and let $P \in T_QC$, which is maybe $Q$. 
Using Hurwitz formula for the projection $\hat{\pi}_{P}$, the ramification divisor has degree $4$ (resp. $2$) if $P \ne Q$ (resp. if $P=Q$). 
Hence one of the following conditions holds. 
\begin{itemize}
\item[(1)] There exist exactly two flexes whose tangent lines contain $P$. 
\item[(2)] There exist exactly one flex and two other points whose three tangent lines contain $P$. 
\item[(3)] $P=Q$ and there exist a unique point $R$ which is not flex and the tangent line at the point contains $P$. 
\end{itemize} 
The former assertion is derived from an elementary argument of projective duality (see, for example, \cite[Lemma 2.1]{fukasawa-miura}). 
We consider (3). 
Now, $\gamma_C(P)=Q_1$. 
Let $\gamma_C(R)=Q_2$. 
Assume $Q_1 \in T_{Q_2}C^*$.  
Then $T_{Q_2}C^*=\overline{Q_1Q_2}=[P]$ and hence $R=\gamma_{C^*}(Q_2)=[T_{Q_2}C^*]=P$ (see \cite[Lemma 2.1]{fukasawa-miura} for the symbol $[*]$). 
This is a contradiction. 

Since $C^{**}=C$ has a unique node, the latter assertion holds.
\end{proof}

\begin{lemma} \label{smooth cubic}
Let $C \subset \mathbb P^2$ be a smooth plane curve of degree three and let $\ell \subset \mathbb{P}^{2*}$ be a line with $\ell \cap {\rm Sing}(C^*)\ne \emptyset$. 
Then, one of the following conditions holds. 
\begin{itemize}
\item[$(1)$] $C^* \cap \ell$ consists of three singular points.  
\item[$(2)$] $C^* \cap \ell$ consists of two singular points and two smooth points. 
\item[$(3$)] $C^* \cap \ell$ consists of one singular point and four smooth points. 
\item[$(4)$] $C^* \cap \ell$ consists of exactly one singular point and three smooth points.
\end{itemize} 
Furthermore, there exists a line $\ell$ with ($1$) if and only if $C$ is projectively equivalent to the Fermat curve. 
\end{lemma}

\begin{proof}
See \cite[Proof of Theorem 1.2]{fukasawa-miura}
\end{proof} 

\begin{lemma} \label{Fermat cubic}
Let $F_3$ be the Fermat curve of degree $d=3$ and let $\ell_{U}, \ell_V, \ell_W \subset \mathbb P^{2*}$ be the lines given by $U=0, V=0, W=0$ respectively. 
Then, ${\rm Sing}(F_3^*)=F_3^* \cap (\ell_U \cup \ell_V \cup \ell_W)$ and sets $F_3^* \cap \ell_U$, $F_3^* \cap \ell_V$, $F_3^* \cap \ell_W$ consist of exactly three points. 
\end{lemma}

\begin{proof}
See \cite[Proof of Theorem 1.2]{fukasawa-miura}. 
\end{proof}

\begin{proof}[Proof of Theorem \ref{cubic}]
We consider (1). 
Let $C$ be a cubic curve which does not have a cusp and let $R \in \mathbb P^{2*}$ be a Galois point. 
If $C$ is smooth, then there exist no Galois points for $C^*$ (\cite[Theorem 1.2]{fukasawa-miura}). 
Therefore, $C$ is a nodal cubic and $d^*=4$. 
Let ${\rm Sing}(C^*)=\{S_1, S_2, S_3\}$, which are cusps coming from three flexes of $C$. 
If $R \in \mathbb P^{2*} \setminus C^*$, then $R$ must lie on the line $\overline{S_iS_j}$ for some $i,j$ with $i \ne j$ by Fact \ref{Galois covering}(2) and Lemma \ref{nodal}. 
By Lemma \ref{nodal}, the line $\overline{S_kR}$ with $k \ne i, j$ contains an unramified smooth point. 
This is a contradiction. 
Therefore, $R \in C^*$. 
If $R \in C_{\rm sm}^*$, then $\hat{\pi}_R$ is a triple Galois covering. 
Therefore, $C^* \cap \overline{RS_i}=\{R, S_i\}$ for each $i$, by Fact \ref{Galois covering}(2). 
Since the degree of ramification divisor is $4$ by Hurwitz formula, this is a contradiction.  

We consider (2). 
Let $C$ be a nodal cubic. 
Assume that $C$ is defined by $X^2Z+G(Y, Z)=0$ and $P=(1:0:0)$. 
We prove that $G_{\overline{P}} \cong D_8$ in this case. 
If $n=1$, then $C$ is defined by $X^2Z+cY^3=0$ for some $c \in K$, which has a cusp. 
Therefore, $n=2$. 
It follows from Lemma \ref{Property of f_P inner}(2) that $r=2$. 
Then the degree of $f_P\circ\pi_P$ is $4$.  
Since $\deg C^*=4$, $\overline{P} \in \mathbb P^{2*} \setminus C^*$. 
By Theorem \ref{inner}(4), we have $|G_{\overline{P}}|=8$. 
Let ${\rm Gal}(L_{\overline{P}}/K(u,v))=\langle \tau \rangle$, where $\tau$ is of order two. 
Since $K(u, v)/K(v)$ is not Galois, ${\rm Gal}(L_{\overline{P}}/K(u,v)) \subset G_{\overline{P}}$ is not a normal subgroup. 
Therefore, there exists $\eta \in G_{\overline{P}}$ such that $\eta^{-1}\tau\eta \ne \tau$. 
Since $G_{\overline{P}}$ is not Abelian, there exists an element $\sigma$ of order four. 
Then $\tau \not\in \langle \sigma \rangle$ or $\eta^{-1}\tau\eta \not\in \langle\sigma \rangle$. 
Therefore, $G_{\overline{P}} \cong D_8$. 
Note that each flex has the standard form $X^2Z+G(Y, Z)=0$ up to a projective equivalence (cf. \cite{miura3}), i.e. which is extendable Galois.  
It is well-known that $C$ has three flexes.  
Therefore, we have three points in $\mathbb P^{2*} \setminus C^*$ whose Galois groups are the dihedral group of order eight. 

Let $R \in \mathbb P^{2*} \setminus C^*$ be a point with $G_R \cong D_8$. 
There exists a subgroup $H \subset G_{R}$ of order $4$ such that ${\rm Gal}(L_R/K(C^*)) \subset H$, and the intermediate field $L_R^H$ of $K(C^*)/\pi_R^*K(\mathbb P^1)$ satisfies $[K(C^*):L_R^H]=2$. 
Then, $\hat{\pi}_R$ is the composite map of two maps $\mathbb P^1 \rightarrow \mathbb P^1$ of degree two. 
Note that a map of degree two over $\mathbb P^1$ has at least two branch points. 
Considering the conditions as in Lemma \ref{nodal}, there are exactly two fibers of $\hat{\pi}_R$ contains exactly two ramification points. 
Let $T$ be the bitangent line as in Lemma \ref{nodal}. 
Then, $R$ is the intersection point $T \cap \overline{S_iS_j}$ for some $i, j$ with $i \ne j$. 
This implies that the number of points $R$ with $G_R \cong D_8$ is at most three. 

We consider (3). 
Let $C$ be smooth and let $P \in C$ be extendable Galois. 
By Proposition \ref{smooth}, $\overline{P} \in \mathbb P^{2*} \setminus C^*$. 
It follows from Lemma \ref{Property of f_P inner}(2) that $r=d=3$. 
By Theorem \ref{inner}(2), $6 \le |G_{\overline{P}}| \le 48$. 
Since $\overline{P}$ is not Galois by Theorem \ref{inner}(3) and $\pi_{\overline{P}}$ has degree $6$, we have $12 \le |G_{\overline{P}}|$. 
As in the proof of Theorem \ref{outer}(2), $L_{\overline{P}}$ is generated by square roots over the Galois closure of a triple covering. 
Therefore, $|G_{\overline{P}}|$ is given by $3 \times 2^i$ or $6 \times 2^i$, where $i=1, 2$ or $3$. 
We have $|G_{\overline{P}}|=12, 24$ or $48$. 
According to \cite[Corollary 2]{miura3}, we have nine extendable Galois points in $C$. 
We have assertion (3).

We consider (4). 
Let $R \in \mathbb P^{2*} \setminus C^*$ be a point with $G_R \cong (\mathbb Z/3\mathbb Z) \times S_3$. 
There exists a subgroup $H \subset G_R$ of order $9$ such that ${\rm Gal}(L_R/K(C^*)) \subset H$, and the intermediate field $L_R^H$ of $K(C^*)/\pi_R^*K(\mathbb P^1)$ satisfies $[K(C^*):L_R^H]=3$. 
Then, $\hat{\pi}_R$ is the composite map of maps of degree three and of degree two. 
Note the degree two map over $\mathbb P^1$ has at least two branch points. 
This implies that there are two lines $\ell \ni R$ with (1) in Lemma \ref{smooth cubic}.  
Then, $C$ is projectively equivalent to the Fermat curve. 
By Lemma \ref{Fermat cubic}, $R$ is the intersection point of two lines of the three lines $\ell_U, \ell_V, \ell_W$. 
This implies that the number of points $R$ with $G_R \cong (\mathbb Z/3\mathbb Z)\times S_3$ is at most three. 
By Proposition \ref{smooth} and Theorem \ref{group}, we have three points $R \in \mathbb P^{2*} \setminus F_3^*$ with $G_R \cong (\mathbb Z/3\mathbb Z) \times S_3$. 
\end{proof}

\end{document}